\documentclass{article}
\usepackage[utf8]{inputenc}
\usepackage{amsmath}
\usepackage{amsthm}
\usepackage{amssymb}
\usepackage{amsfonts}
\usepackage{bm}
\usepackage{bbm}
\usepackage{mathtools}
\usepackage{tikz}
\usepackage{tikz-cd}
\usepackage{slashed}
\usepackage{xfrac}  
\usepackage{tcolorbox}
\usepackage{comment}

\newtheorem{theorem}{Theorem}
\newtheorem{lemma}[theorem]{Lemma}
\newtheorem{fact}[theorem]{Fact}

\newtheorem{corollary}[theorem]{Corollary}
\newtheorem*{definition}{Definition}
\newtheorem*{convention}{Convention}
\newtheorem{question}[theorem]{Question}
\newtheorem{mainquestion}[theorem]{Main Question}
\newtheorem*{remark}{Remark}
\newtheorem*{example}{Example}

\usepackage[style=alphabetic]{biblatex}
\bibliography{Mathbib.bib}

\title{A note on ded $\kappa$ and the part-whole principle\thanks{The author thanks Nicholas Ramsey, Joel David Hamkins, and Gabriel Goldberg for discussion. Forthcoming in \textit{Philosophy of Science}.}}
\author{Yuanshan Li}

\begin{document}

\maketitle

\begin{abstract}
    We establish a connection between the part-whole principle and the quantity $\text{ded}$ $\kappa$ -- a generalized cardinal characteristic related to the number of Dedekind cuts of a linear order. As consequences, we improve a result of Mancosu and Massas \cite{mancosu2024totality} on generalized probability functions and propose some questions. 
\end{abstract}

\section{Introduction}

We recall Definitions 2.1 and 2.2 from Mancosu and Massas \cite{mancosu2024totality}:

\begin{definition}
    A generalized probability range is a tuple $(V, \le, +, 0)$ such that:
    \begin{itemize}
        \item $(V, \le)$ is a partial order with a least element $0$;
        \item $(V, +, 0)$ is a commutative monoid;
        \item if $x, y, z\in V$ and $x<y$ then $x+z<y+z$.
    \end{itemize}
\end{definition}

\begin{definition}
    Let $(V, \le, +, 0)$ be a generalized probability range. A generalized $V$-probability space is a tuple $(\Omega, \mathcal{F}, \mu)$ such that:
    \begin{itemize}
        \item $\mathcal{F}\subseteq 2^\Omega$ is an algebra;
        \item $\mu: \mathcal{F}\rightarrow V$ satisfies $\mu(A)\ge 0$ for all $A\in\mathcal{F}$ and $\mu(A\sqcup B)=\mu(A)+\mu(B)$ for disjoint $A, B\in\mathcal{F}$.
    \end{itemize}
    In this context, the function $\mu$ is to be called a generalized probability function, or a generalized probability. 
\end{definition}

The main objects of interest in \cite{mancosu2024totality} are generalized probabilities that are both \textit{total} and \textit{regular}, defined as follows:

\begin{definition}
    A generalized probability $\mu$ is total if its domain $\mathcal{F}=2^\Omega$.
\end{definition}

\begin{definition}
    A generalized probability $\mu$ is regular if $\mu(A)>0$ for all nonempty $A\in\mathcal{F}$.
\end{definition}

One of the main mathematical issues investigated in \cite{mancosu2024totality} is the extent to which the existence of regular total generalized probabilities can be characterized in terms of the cardinalities of the associated sets $\Omega$ and $V$. Hofweber and Schindler \cite{hofweber2016hyperreal} obtained the following result via an elementary compactness argument:

\begin{theorem}[Hofweber and Schindler]
    For any $\Omega$ there is a hyperreal field\footnote{Note that a hyperreal field is a model of the axioms of real closed fields, hence satisfies Definition 1.} $\mathbb{R}^*$ with $|\mathbb{R}^*|\le2^{|\Omega|}$ and a regular total generalized probability function $\mu: 2^\Omega\rightarrow \mathbb{R}^*$.
\end{theorem}

Theorem 3.1 of \cite{mancosu2024totality}, an improvement of Corollary 1 of Pruss \cite{pruss2013probability}, states the following:

\begin{theorem}[Mancosu and Massas]
    If $|V|\le|\Omega|$ then there is no regular total generalized probability function $\mu: 2^\Omega\rightarrow V$.
\end{theorem}

Recall that the continuum hypothesis ($\mathsf{CH}$) states that the cardinality of the powerset of $\omega$ is the first uncountable cardinal, i.e. $2^\omega=\aleph_1$; the generalized continuum hypothesis ($\mathsf{GCH}$) at a cardinal $\kappa$ states that the cardinality of the powerset of $\kappa$ is its successor cardinal, i.e. $2^\kappa=\kappa^+$. By Theorems 1 and 2, it follows that:

\begin{corollary}[Mancosu and Massas]
    If $\mathsf{GCH}$ holds at $|\Omega|$, then there exists a regular total generalized probability function $\mu: 2^\Omega\rightarrow V$ iff $|V|\ge2^{|\Omega|}$.
\end{corollary}

Corollary 3 invites a natural question, namely whether the assumption of $\mathsf{GCH}$ can be removed, as it is well-known that various failures of $\mathsf{GCH}$ are consistent with $\mathsf{ZFC}$. 

\begin{mainquestion}
    Is it consistent with $\mathsf{ZFC}$ that there is a regular total generalized probability function $\mu: 2^\Omega\rightarrow V$ with $|V|< 2^{|\Omega|}$?
\end{mainquestion}

In this note we make some very modest clarifications on the Main Question: first, we improve Theorem 2 by weakening the hypothesis $|V|\le|\Omega|$, sharpening the existing results and leading to some necessary conditions that the candidate set-theoretic models for the Main Question must satisfy. Second, this is done merely using the order structure on $V$, which avoids the complications introduced by the additive structure.

To illustrate the content of our improvement, let us consider the special case where $\Omega$ is countable. In this case, the Main Question reduces to the following.

\begin{question}
    Is it consistent with $\mathsf{ZFC}$ that there is a regular total generalized probability function $\mu: 2^\omega\rightarrow V$ with $|V|< 2^\omega$?
\end{question}

By Corollary 3, there is no such function if $\mathsf{CH}$ holds. However, Corollary 3 does not settle the question when the cardinality of continuum is large. For example, if $2^\omega=\aleph_2$, is there a regular total generalized probability function for subsets of $\omega$ whose range has cardinality $\aleph_1$? As the reader shall see, it follows from our results in this note that the answer to Question 5 is negative. 

\section{Main results}

The part-whole principle states that if $A, B$ are sets and $A\subsetneq B$, then the size of $A$ is strictly less than the size of $B$. This principle is of independent interest in set theory and its philosophy. It motivates the following definition:

\begin{definition}
    A part-whole size function (for subsets of $\kappa$) is an order homomorphism $\mu: (2^\kappa, \subsetneq)\rightarrow (V, <)$, where $(V, <)$ is a strict partial order.\footnote{One might think that the sizes of any two sets should be comparable, so it may appear natural to require the range of size functions to be linearly ordered. Here we don't make this requirement to coordinate with generalized probability ranges being merely partially ordered. All results in our note remain valid if we impose linearity in both definitions.}
\end{definition}

Part-whole size functions are closely related to the previous setting in the following way:

\begin{lemma}
    If $\mu$ is a regular total generalized probability function then its reduct $\mu: (2^\Omega, \subsetneq)\rightarrow (V, <)$ is a part-whole size function for subsets of $\Omega$.
\end{lemma}
\begin{proof}
    Let $A\subsetneq B$, so $C=B\setminus A$ is nonempty. By regularity, $\mu(C)>0$. By finite additivity, $\mu(A)+\mu(C)=\mu(B)$. We have for $x, y, z\in V$, $x<y$ implies $x+z<y+z$. In particular, $0<\mu(C)$ implies 
    \[0+\mu(A)<\mu(C)+\mu(A)=\mu(B).\]
    Hence $\mu(A)<\mu(B)$. 
\end{proof}

From now on, the chief objects of interest would be part-whole size functions. For convenience, we make the following convention:

\begin{convention}
    We shall identify the base space $\Omega$ with a cardinal $\kappa$. 
\end{convention}

We note that once the additive structure is removed, it already makes the proof of Theorem 2 much simpler than the proof written down in \cite{mancosu2024totality}, even following their existing approach, which involves the following fixed point theorem attributed to Zermelo.\footnote{For an informative discussion of Zermelo's result and its context, see Kanamori \cite{kanamori1997mathematical}, where it appears as Corollary 2.3.}

\begin{fact}[Zermelo]
    For any function $F: 2^\kappa \rightarrow\kappa$ there is $A \subsetneq B$ s.t. $F(A) = F(B)$.
\end{fact}

It readily follows that there can be no part-whole size function if $|V|\le\kappa$: suppose $\mu: 2^\kappa\rightarrow V$ is such a function, fix an injection $i: V\rightarrow\kappa$. By Zermelo's theorem applied to $i\circ \mu$, there is $A\subsetneq B$ such that $i(\mu(A))=i(\mu(B))$. Since $i$ is injective, $\mu(A)=\mu(B)$, but $A\subsetneq B$, contradicting part-whole. By Lemma 6, if $|V|\le\kappa$ there can't be regular total generalized probability functions as well.

We now state our main result, from which we can prove a stronger version of Theorem 2 without using Zermelo's fixed point theorem.

The key observation is the following: by the part-whole principle, if $\mu$ is a part-whole size function and $A\subsetneq B$, then $\mu(A)<\mu(B)$, in particular they are different values in $V$. In general, if $(A_i: i\in I)$ forms a strictly increasing chain of subsets of $\kappa$, meaning that $I$ is a linear order and if $i<_Ij$ then $A_i\subsetneq A_j$, we need at least $|I|$-many values in $V$. Therefore the supremum of the lengths of these strictly increasing chains of subsets gives a lower bound of the number of values needed. It turns out that this quantity can be identified with $\text{ded }\kappa$, a generalized cardinal characteristic that has been independently studied in mathematical logic.

\begin{definition}
    For a linear order $(I, <)$, a Dedekind cut is a pair $(L, R)$ such that $L, R$ are disjoint subsets of $I$, $I=L\sqcup R$, and $i<j$ for all $i\in L$ and $j\in R$. In this setting, $L$ is called a left cut and $R$ is called a right cut.
\end{definition}

\begin{definition}
    For a cardinal $\kappa$, $\text{ded }\kappa=$
    \[\sup\{\lambda: \text{ there is a linear order of cardinality }\kappa\text{ with }\lambda\text{-many Dedekind cuts}\}.\]
\end{definition}

\begin{theorem}
    The following are equivalent: i) there is a linear order of size $\kappa$ with $\lambda$-many cuts; ii) there is a $\subsetneq$-chain of subsets of $\kappa$ with size $\lambda$.
\end{theorem}

\begin{proof}
    Clearly $i)\Rightarrow$ ii), since the collection of left cuts forms a $\subsetneq$-chain.
    
    Conversely, if there is a $\subsetneq$-chain of subsets of $\kappa$ with cardinality $\lambda$, then we can endow any set $A$ that has cardinality $\kappa$ with a linear order structure such that it has $\lambda$-many cuts. To see this, let $(A_i: i\in I)$ be a $\subsetneq$-chain, where $A_i\subseteq A$, and $I$ is a linear order with cardinality $\lambda$. For each $x\in A$, define
    \[I_x=\{i\in I: x\not\in A_i\}.\]
    We order $A$ by $x<_Ay$ iff $I_x\subsetneq I_y$, i.e. there exists $i\in I$ s.t. $x\in A_i$ but $y\not\in A_i$. Intuitively, this means that $x$ ``shows up earlier'' in the chain $(A_i: i\in I)$ than $y$. If $I_x=I_y$ then the $<_A$-order of $x, y$ is arbitrary. Clearly the order $(A, <_A)$ is a linear order of size $\kappa$. Moreover, it has at least $\lambda$-many cuts. To see this, we verify that for each $i\in I$, $A_i$ is a different left cut:

    1. if $y\in A_i$ and $x<_Ay$, then $I_x\subsetneq I_y$. Since $y\in A_i$, $i\not\in I_y$, hence $i\not\in I_x$, so $x\in A_i$. This shows that $A_i$ is downward closed and so is a left cut. 

    2. let $i<_Ij$, then $A_i\subsetneq A_j$, so there is $x\in A_j\setminus A_i$. This shows that $A_i$, $A_j$ are different cuts.
\end{proof}

\begin{remark}
    Note that the above construction involves certain fragments of the Axiom of Choice ($\mathsf{AC}$). The order extension principle states that every partial order can be extended to a linear order. It is a consequence of $\mathsf{AC}$ but is not provable from $\mathsf{ZF}+$``every set can be linearly ordered'', by a result of Mathias \cite{mathiasorderextension}. The above construction can be understood as first defining the order $<_A$ as a partial order given by the ``shows up earlier than'' relation, and then extend it to a linear order using the order extension principle. Also note that Hofweber and Schindler's proof of Theorem 1 uses the compactness of first order logic, which is equivalent to the Boolean prime ideal theorem, also a weak form of $\mathsf{AC}$. Here we remind the reader that we take the basic setting of this note to be $\mathsf{ZFC}$.
\end{remark}

The following corollary follows from Theorem 8.

\begin{corollary}
    If $|V|<\text{ded }\kappa$, then there is no part-whole size function $\mu: (2^\kappa, \subsetneq)\rightarrow (V, <)$. In particular, there is no regular total generalized probabilities.
\end{corollary}
\begin{proof}
    Fix a cardinal $\lambda$ by cases:
    
    1. if $\text{ded }\kappa$ is realized, meaning that the set
    
    \[S=\{\lambda: \text{ there is a linear order of size }\kappa\text{ with }\lambda\text{-many cuts}\}\]
    
    contains a maximum (which must be equal to $\text{ded }\kappa$), let $\lambda=\text{ded }\kappa$.
    
    2. if $\text{ded }\kappa$ is not realized, then $S$ is unbounded in $\text{ded }\kappa = \sup S$, so there is $\lambda\in S$ s.t. $|V|<\lambda<\text{ded }\kappa$.

    Either way, we have $|V|<\lambda\in S$. By ii) of Theorem 7, there is a $\subsetneq $-chain of subsets of $\kappa$ with length $\lambda$. If $\mu$ exists then it assigns different values for these sets, we would need $|V|\ge\lambda$, contradiction! The ``in particular'' part follows from Lemma 6.
\end{proof}

\section{Remarks and questions}

We first observe that Corollary 9 strictly improves Theorem 2, by the following fact.\footnote{For a reference see e.g. Chernikov and Shelah \cite{Chernikov_2015}.}

\begin{fact}
    \[\kappa<\text{ded }\kappa \le(\text{ded }\kappa)^{\aleph_0}\le 2^\kappa.\]
\end{fact}

In particular, $\kappa<\text{ded }\kappa$.\footnote{To see this, let $\mu$ be the least such that $2^\mu >\kappa$, then $\mu\le\kappa$. Since $2^{<\mu}=\bigcup_{\lambda<\mu}2^\lambda$ and any $2^\lambda\le\kappa$ by the choice of $\mu$, we have $2^{<\mu}\le \kappa$, and it is a tree with $2^\mu > \kappa$ many paths.} Hence, our Corollary 9 strictly improves Theorem 2 in situations where $\text{ded } \kappa\ge \kappa^{++}$. This situation is indeed possible:

\begin{example}
    $\text{ded }\omega=2^\omega$, since $\mathbb{Q}$ has $2^\omega$-many Dedekind cuts. If $\mathsf{CH}$ fails, i.e. $2^\omega\ge\aleph_2$, then this is a situation where $\text{ded } \kappa\ge \kappa^{++}$. In this case, Corollary 9 tells us that there can't be part-whole size functions or regular total generalized probabilities for subsets of $\omega$ whose range has cardinality less than the continuum. This is an information that is not available from Theorem 2.
\end{example}

Thus the above example, in particular, answers Question 5 in the negative. Our improved version of Corollary 3 is the following: 

\begin{corollary}
    If $\text{ded }\kappa=2^\kappa$, then there exists a regular total probability function $\mu: 2^\kappa\rightarrow V$ iff $|V|\ge 2^\kappa$.
\end{corollary}

\begin{remark}
In particular we know that in the following situations, Main Question 4 admits negative answers:
\begin{itemize}
    \item $\text{ded }\omega=2^\omega$ so this can't happen at $\omega$. 
    \item $2^{<\kappa}$ has $2^\kappa$-many paths, so $2^{<\kappa}=\kappa$ implies $\text{ded }\kappa=2^\kappa$. 
    \item If there exists $n$ such that $2^\kappa=\kappa^{+n}$ then $\text{ded }\kappa=2^\kappa$. (A result of Baumgartner, see Chernikov and Shelah \cite{Chernikov_2015} or Mitchell \cite{mitchell1972aronszajn}.)
\end{itemize}
\end{remark}

On the other hand, a positive answer to Main Question 4 is not ruled out. The following result is due to Mitchell \cite{mitchell1972aronszajn}:

\begin{theorem}[Mitchell]
    If $\text{cf }\kappa\ge\omega_1$, then it is relative consistent with $\mathsf{ZFC}$ that $\text{ded }\kappa<2^\kappa$. 
\end{theorem}

We end this note with the following questions. To the knowledge of the author, they are open. Question 13 sharpens Question 4. Question 14 asks whether we can prove a certain version of a converse to Lemma 6.

\begin{question}
    Assume $\text{ded }\kappa<2^\kappa$, is it consistent that there is a part-whole size function $\mu: (2^\kappa, \subsetneq) \rightarrow (V, <)$ with $|V|<2^\kappa$? If so, are there regular total generalized probabilities?
\end{question}

\begin{question}
    Let $(V, <)$ be a partial order and let $\mu: 2^\kappa\rightarrow (V, <)$ be a part-whole size function. Does it follow that there is a regular total probability $\Tilde{\mu}: 2^\kappa\rightarrow (\Tilde{V}, 0, +, <)$ such that $|\tilde{V}| = |V|$?
\end{question}

We conclude with a remark on Question 14. If we are given a part-whole size function $\mu$ and attempt to lift it to $\Tilde{\mu}$, we face the following obstacle: \textit{a priori}, it seems consistent with the part-whole principle that there exist disjoint sets $A, B, C$ such that $\mu(A)=\mu(B)$ but $\mu(A\cup C)<\mu(B\cup C)$. In this case, if we define an additive structure on the range of $\mu$, it fails the third condition in the definition of a generalized probability range.

\printbibliography

\end{document}